\documentclass[11pt,a4paper]{article}
\usepackage[utf8]{inputenc}

\usepackage{latexsym}
\usepackage[margin=3cm]{geometry}
\usepackage[dvipdfmx]{graphicx}

\usepackage{hyperref}
\usepackage{mymacro}
\usepackage{mathtools}
\usepackage[
  backend=biber,
     style=ext-alphabetic,
    maxbibnames=5,
    sorting=nyt,
    giveninits,
    articlein=false,
    url=false,  
    doi=false,
    eprint=true,
    isbn=false
]{biblatex}
\renewcommand{\labelenumi}{$\mathrm{(\roman{enumi})}$}
\renewcommand{\labelenumii}{$\mathrm{(\alph{enumii})}$}

\makeatletter

\@addtoreset{equation}{section}
\makeatother

\title{A note on Hamiltonian stability for sheaves
}
\author{Tomohiro Asano \and Yuichi Ike\thanks{Supported by JSPS KAKENHI (21K13801 and 22H05107).}}
\date{\today}

\begin{document}
\maketitle

\begin{abstract}
    We show a strong Hamiltonian stability result for a simpler and larger distance on the Tamarkin category.
    We also give a stability result with support conditions.
\end{abstract}

\section{Introduction}

In \cite{AI20}, the authors introduced an interleaving-like distance $d_{\mathrm{int}}$ on the Tamarkin category by using the notion of ``$(a,b)$-interleaved" and proved a Hamiltonian stability theorem (\cite[Thm.~4.16]{AI20}) for the distance. 
Moreover, in \cite{AI22completeness}, we considered a distance $d_{\mathrm{w\text{-}isom}}$ on the Tamarkin category by using ``weakly $(a,b)$-isomorphic" and showed that the stability holds also for $d_{\mathrm{w\text{-}isom}}$ (\cite[Thm.~3.14]{AI22completeness}). 
This is a stronger result than the result in \cite{AI20} since $d_{\mathrm{int}} \le d_{\mathrm{w\text{-}isom}}$.
We can also equip the Tamarkin category with a distance $d_{\mathrm{isom}}$ by using ``$(a,b)$-isomorphic" (see, for example, \cite{AI22completeness} and \cite{guillermou2022gamma}), which satisfies the inequalities
\begin{equation}
    d_{\mathrm{w\text{-}isom}} \le d_{\mathrm{isom}} \le 2 d_{\mathrm{w\text{-}isom}}.
\end{equation}
In this note, we show that the Hamiltonian stability result holds for $d_{\mathrm{isom}}$.

Since we introduced several notions of distances on the Tamarkin category as above, sometimes it is confusing how to choose a distance to obtain the Hamiltonian stability result. 
The first aim of this note is to explicitly give the strongest stability result and state that we only need the distance $d_{\mathrm{isom}}$ in the end. 

The second aim is to give a Hamiltonian stability result with support conditions. 
We also prove an estimate not only by the oscillation norm $\|H\|_{\mathrm{osc}}$ of a Hamiltonian function $H$, but also by the oscillation norm $\|H\|_{\mathrm{osc},A}$ of $H$ restricted to a non-empty closed subset $A$, in the context of the sheaf-theoretic energy estimate.

\section{Preliminaries}

Let $X$ be a connected $C^\infty$-manifold without boundary.
We let $\pi \colon T^*X \to X$ denote its cotangent bundle and write $(x;\xi)$ for its local homogeneous coordinate system. 
We also let $0_X$ denote the zero-section of $T^*X$ and set $\rT X \coloneqq T^*X \setminus 0_X$. 
Throughout this note, we fix a field $\bfk$. 
We denote by $\SD(\bfk_X)$ the unbounded derived category of sheaves of $\bfk$-vector spaces on $X$.
In this note, we freely use the notions in the microlocal theory of sheaves~\cite{KS90} (see also \cite{Gu19}).
In particular, for an object $F \in \SD(\bfk_X)$, we let $\MS(F) \subset T^*X$ denote its microsupport.
We set $\rMS(F) \coloneqq \MS(F) \cap \rT X$.

\subsection{Distances on the Tamarkin category}

We write $(t;\tau)$ for the homogeneous coordinate system on the cotangent bundle $T^*\bR_t$.
It is prove by Tamarkin~\cite{Tamarkin} that the functor $P_l \coloneqq \bfk_{X \times [0,\infty)} \star (\ast) \colon \SD(\bfk_{X \times \bR_t}) \to \SD(\bfk_{X \times \bR_t})$ defines a projector onto ${}^{\bot} \SD_{\{\tau \le 0\}}(\bfk_{X \times \bR_t})$.
Here $\SD_{\{\tau \le 0\}}(\bfk_{X \times \bR_t})$ is the triangulated full subcategory consisting of the objects microsupported in $\{ \tau \le 0 \}= \{(x,t;\xi,\tau) \mid \tau \le 0 \} \subset T^*(X \times \bR_t)$ and ${}^\bot (\ast)$ denotes the left orthogonal.
The \emph{Tamarkin category} $\cD(X)$ is defined as the left orthogonal:
\begin{equation}
    \cD(X) \coloneqq {}^{\bot} \SD_{\{\tau \le 0\}}(\bfk_{X \times \bR_t}).
\end{equation}
For $c \in \bR$, let $T_c \colon X \times \bR_t \to X \times \bR_t, (x,t) \mapsto (x,t+c)$ denote the translation map to the $\bR_t$-direction by $c$. 
We simply write $T_c F$ to mean ${T_c}_*F$ for an object $F \in \cD(M)$.
Then, for $F \in \cD(X)$ and $c,d \in \bR_{\ge 0}$ with $c \le d$, one can get a canonical morphism $\tau_{c,d}(F) \colon T_c F \to T_d F$. 
By using the canonical morphism, we define several relations on pairs of objects of $\cD(X)$ as follows.

\begin{definition}
    Let $F,G \in \cD(X)$ and $a,b \in \bR_{\ge 0}$.
    \begin{enumerate}
        \item The pair $(F,G)$ is said to be \emph{$(a,b)$-isomorphic} if there exist morphism $\alpha \colon F \to T_a G$ and $\beta \colon G \to T_b F$ such that 
        \begin{enumerate}
        \renewcommand{\labelenumii}{$\mathrm{(\arabic{enumii})}$}
            \item the composite $F \xrightarrow{\alpha} {T_a} G \xrightarrow{{T_a} \beta} {T_{a+b}} F$ is equal to $\tau_{0,a+b}(F) \colon F \to {T_{a+b}} F$ and 
	        \item the composite $G \xrightarrow{\beta} {T_b} F \xrightarrow{{T_b} \alpha} {T_{a+b}} G$ is equal to $\tau_{0,a+b}(G) \colon G \to {T_{a+b}} G$.
        \end{enumerate}
        One defines 
        \begin{equation}
            d_{\mathrm{isom}}(F,G) \coloneqq \inf \lc a+b \relmid \text{$(F,G)$ is $(a,b)$-isomorphic} \rc. 
        \end{equation}
        \item The pair $(F,G)$ of objects of $\cD(X)$ is said to be \emph{weakly $(a,b)$-isomorphic} if there exist morphisms $\alpha, \delta \colon F \to {T_a} G$ and $\beta, \gamma \colon G \to {T_b} F$ such that
        \begin{enumerate}
        	\renewcommand{\labelenumii}{$\mathrm{(\arabic{enumii})}$}
        	\item $F \xrightarrow{\alpha} {T_a} G \xrightarrow{{T_a} \beta} {T_{a+b}} F$ is equal to $\tau_{0,a+b}(F) \colon F \to {T_{a+b}} F$,
        	\item $G \xrightarrow{\gamma} {T_b} F \xrightarrow{{T_b} \delta} {T_{a+b}} G$ is equal to $\tau_{0,a+b}(G) \colon G \to {T_{a+b}} G$, and
        	\item $\tau_{a,2a}(G) \circ \alpha=\tau_{a,2a}(G) \circ \delta$ and $\tau_{b,2b}(F) \circ \beta=\tau_{b,2b}(F) \circ \gamma$.
        \end{enumerate}
        One defines 
        \begin{equation}
            d_{\mathrm{w\text{-}isom}}(F,G) \coloneqq \inf \lc a+b \relmid \text{$(F,G)$ is weakly $(a,b)$-isomorphic} \rc. 
        \end{equation}
    \end{enumerate}
\end{definition}

The above terminology is the same as \cite{AI22completeness} and different from \cite{AI20}.
It is also similar to the terminology in \cite{KS18persistent}: a pair $(F,G)$ of objects of $\cD(\pt)$ is $(a,a)$-isomorphic if and only if $F$ and $G$ are $a$-isomorphic in the sense of \cite{KS18persistent}.
The notion of ``$(a,b)$-interleaved" is also defined as follows:
a pair $(F,G)$ of objects of $\cD(X)$ is said to be \emph{$(a,b)$-interleaved} if there exist morphisms $\alpha, \delta \colon F \to {T_a} G$ and $\beta, \gamma \colon G \to {T_b} F$ satisfying the conditions (1) and (2) above. 
We can also define 
\begin{equation}
    d_{\mathrm{int}}(F,G) \coloneqq \inf \lc a+b \relmid \text{$(F,G)$ is $(a,b)$-interleaved} \rc.
\end{equation}
Then, we have the chain of inequalities 
\begin{equation}
    d_{\mathrm{int}} \le d_{\mathrm{w\text{-}isom}} \le d_{\mathrm{isom}} \le 2 d_{\mathrm{w\text{-}isom}}.
\end{equation}

\subsection{Sheaf quantization of Hamiltonian isotopies}\label{subsec:SQ_ham}

For the later applications, we first state the main result of \cite{GKS} in a general form. 
Let $N$ be a connected non-empty manifold and $W$ a contractible open subset of $\bR^n$ with the coordinate system $(w_1,\dots,w_n)$ containing $0$. 
Let us consider $\psi =(\psi_w)_{w \in W} \colon \rT N \times W \to \rT N$ be a homogeneous Hamiltonian isotopy, that is, a $C^\infty$-map satisfying (1)~$\psi_w$ is homogeneous symplectic isomorphism for each $w \in W$ and (2)~$\psi_0=\id_{\rT N}$.
We can define a vector-valued homogeneous function $h \colon \rT N \times W \to \bR^n$ by $h=(h_1,\dots,h_n)$ with 
\begin{equation}
    \frac{\partial \psi_w}{\partial w_i} \circ \psi_{w}^{-1} = X_{h_i(\bullet,w)},
\end{equation}
where $X_{h_i(\bullet,w)}$ is the Hamiltonian vector field of the function $h_i(\bullet,w) \colon \rT N \to \bR$.
By using the function $h$, we define a conic Lagrangian submanifold $\Lambda_\psi$ of $\rT (N^2 \times W)$ by 
\begin{equation}
    \Lambda_\psi \coloneqq \lc \lb \psi_w(y;\eta), (y;-\eta), (u;-h(\psi_w(y;\eta),w)) \rb \relmid (y;\eta) \in \rT N, w \in W \rc
\end{equation}
The main theorem of \cite{GKS} is the following.

\begin{theorem}[{\cite[Thm.~3.7 and Rem.~3.9]{GKS}}]\label{thm:GKSmain}
    Let $\psi \colon \rT N \times W \to \rT N$ be a homogeneous Hamiltonian isotopy and set $\Lambda_\psi$ as above. 
    Then there exists a unique simple object $\tl{K} \in \SD(\bfk_{N^2 \times W})$ such that $\rMS(\tl{K})=\Lambda_{\psi}$ and $\tl{K}|_{N^2 \times \{0\}} \simeq \bfk_{\Delta_{N}}$.
\end{theorem}

For a non-homogeneous compactly supported Hamiltonian isotopy, we can associate a sheaf by homogenizing the isotopy, which we will explain below. 
In what follows until the end of this note, let $M$ be a connected $C^\infty$ manifold without boundary and  $I$ an open interval of $\bR$ containing the closed interval $[0,1]$.
We write $(x;\xi)$ for a local homogeneous coordinate system on $T^*M$.

\paragraph{1-parameter case}
We say that a $C^\infty$-function $H=(H_s)_{s \in I} \colon T^*M \times I \to \bR$ is \emph{timewise compactly supported} if $\supp(H_s)$ is compact for any $s \in I$. 
A compactly supported \emph{Hamiltonian isotopy} is a flow of the Hamiltonian vector field of a timewise compactly supported function $H$. 
Let $\phi^H=(\phi^H_s)_{s\in I} \colon T^*M \times I \to T^*M$ be the compactly supported Hamiltonian isotopy associated with a timewise compactly supported function $H \colon T^*M \times I \to \bR$.
We define $\wh{H}=(\wh{H}_s)_{s \in I} \colon \rT(M \times \bR_t) \times I \to \bR$ by 
\begin{equation}
    \wh H_s(x,t;\xi,\tau) \coloneqq 
    \begin{cases}
        \tau H_s(x;\xi/\tau) & (\tau\neq 0)\\
        0 & (\tau =0)
    \end{cases}
\end{equation}
and $\wh \phi=(\wh \phi_s)_{s \in I} \colon \rT(M \times \bR_t) \times I \to \rT(M \times \bR_t)$ to be the homogeneous Hamiltonian isotopy of $\wh H$.
Then we can apply \cref{thm:GKSmain} to the homogeneous Hamiltonian isotopy $\wh{\phi}$ and obtain a simple object $\tl{K}$ satisfying $\rMS(\tl{K})=\Lambda_{\wh{\phi}}$.
Since $h=\wh{H} \colon \rT(M\times \bR_t) \times I \to \bR$ in this case, we have
 \begin{equation}
\begin{aligned}
    \Lambda_{\wh{\phi}}
    \coloneqq 
    & 
    \left\{
    \left(
    \wh{\phi}_s(x,t;\xi,\tau), (x,t;-\xi,-\tau), (s;-\tau H_s(\phi_s(x;\xi/\tau)) \right)
    \; \middle| \; \right. \\
    & \hspace{150pt} \left. (x,t;\xi,\tau) \in \rT(M \times \bR_t), 
    s \in I
    \right\}.
\end{aligned}
\end{equation}

Define $q \colon (M \times \bR)^2 \times I \to M^2 \times \bR_t \times I, (x_1,t_1,x_2,t_2,s) \mapsto (x_1,x_2,t_1-t_2,s)$ and set 
\begin{equation}
    q_\pi q_d^{-1}(\Lambda_{\wh{\phi}}) 
    \subset \rT(M^2 \times \bR_t \times W).
\end{equation}
Then the inverse image functor $q^{-1}$ gives an equivalence (see \cite[Cor.~2.3.2]{Gu19})
\begin{equation}\label{eq:equivalence}
    \{ K \in \SD(\bfk_{M^2 \times \bR_t \times W}) \mid \rMS(K) = q_\pi q_d^{-1}(\Lambda_{\wh{\phi}}) \} 
    \simto 
    \{ \tl{K} \in \SD(\bfk_{(M \times \bR)^2 \times W}) \mid \rMS(K) = \Lambda_{\wh{\phi}} \}.
\end{equation}
Recall that we have a projector $P_l \colon \SD(\bfk_{M^2 \times \bR_t}) \to {}^\perp \SD_{\{\tau \le 0\}}(\bfk_{M^2 \times \bR_t})=\cD(M^2)$ and similarly for $\cD(M^2 \times I)$.

\begin{definition}
    Let $H \colon T^*M \times I \to \bR$ be a timewise compactly supported function.
    \begin{enumerate}
        \item One defines $K^H \in \SD(\bfk_{M^2 \times \bR_t \times I})$ to be the object such that $\rMS(K^H)=q_\pi q_d^{-1}(\Lambda_{\wh{\phi}})$ and $K^H|_{M^2 \times \bR_t \times \{0\}} \simeq \bfk_{\Delta_M \times \{0\}}$, that is, the object $K^H$ satisfying $q^{-1}K^H \simeq \tl{K}^H$. 
        One also sets $K^H_s \coloneqq K^H|_{M^2 \times \bR_t \times \{s\}}$ for simplicity.
        \item One defines $\cK^H_s \coloneqq P_l(K^H_s) \in \cD(M^2)$ and $\cK^H \coloneqq P_l(K^H) \in \cD(M^2 \times I)$. 
        The object $\cK^H$ is called the \emph{sheaf quantization} of the isotopy $\phi^H$ or associated with $H$.
    \end{enumerate}
\end{definition}

For a timewise compactly supported function $H \colon T^*M \times I \to \bR$, its oscillation norm is defined as 
\begin{equation}
    \|H\|_{\mathrm{osc}} \coloneqq \int_0^1 \left(\max_{p \in T^*M} H_s(p) - \min_{p \in T^*M} H_s(p) \right) ds.
\end{equation}
The weak stability theorem proved in the previous paper is the following.
 
\begin{theorem}[{\cite[Thm.~3.14]{AI22completeness}, cf.\ \cite[Thm.~4.16]{AI20}}]\label{theorem:SQ_inequality}
    Let $H \colon T^*M \times I \to \bR$ be a timewise compactly supported function. 
    Then $d_{\mathrm{w\text{-}isom}}(\cK^0_1,\cK^H_1) \le \|H\|_{\mathrm{osc}}$.
\end{theorem}

\paragraph{2-parameter case}
In the later application, we also use a 2-parameter Hamiltonian isotopy. 
Let $(G_{s',s})_{(s',s) \in I^2}$ be a 2-parameter family of compactly supported smooth functions on $T^*M$. 
A 2-parameter family of diffeomorphisms $(\phi_{s',s})_{(s',s)\in I^2}$ is determined by $\phi_{s',0}=\id_{T^*M} $ and $\frac{\partial \phi_{s',s} }{\partial s}\circ \phi_{s',s}^{-1}=X_{G_{s',s}}$, where $X_{G_{s',s}}$ is the Hamiltonian vector field corresponding to the function $G_{s',s}$. 
We set $\widehat{G}_{s',s}(x,t;\xi,\tau) \coloneqq \tau G_{s',s}(x;\xi/\tau)$ and define 
a 2-parameter homogeneous Hamiltonian isotopy $\wh{\phi}=(\widehat{\phi}_{s',s})_{s',s}$ by 
\begin{equation}
    \begin{cases}
        \widehat{\phi}_{s',0}=\id_{\rT(M \times \bR_t)}, \\
        \frac{\partial \widehat{\phi}_{s',s}}{\partial s} \circ \widehat{\phi}_{s',s}^{-1} = X_{\widehat{G}_{s',s}}.
    \end{cases}
\end{equation}
Then, we have
\begin{equation}
\begin{aligned}
    \Lambda_{\wh{\phi}} = \left\{ \lb \widehat{\phi}_{s',s}(y;\eta), (y;-\eta), (s'; -\widehat{F}_{s',s}(\widehat{\phi}_{s',s}(y;\eta))), (s;- \widehat{G}_{s',s}(\widehat{\phi}_{s',s}(y;\eta))) \rb \relmid \right. \\ 
    \left. (y;\eta) \in \rT(M \times \bR), s',s \in I \right\},
\end{aligned}
\end{equation}
where the 2-parameter family of homogeneous functions $(\widehat{F}_{s',s})_{s',s}$ is determined by $\frac{\partial \widehat{\phi}_{s',s}}{\partial s'} \circ \widehat{\phi}_{s',s}^{-1}=X_{\widehat{F}_{s',s}}$.
By the construction of $\widehat{\phi}$, there exists a 2-parameter family of compactly supported function $(F_{s',s})_{(s',s) \in I^2}$ satisfying $\widehat{F}_{s',s}(x,t;\xi,\tau)=\tau F_{s',s}(x;\xi/\tau)$ and $\frac{\partial \phi_{s',s} }{\partial s'}\circ \phi_{s',s}^{-1}=X_{F_{s',s}}$.
A calculation in \cite{polterovich2012geometry} or \cite{Oh02normalization} (see also \cite{banyaga1978structure}) shows that 
\begin{equation}\label{eq:another_function}
    \frac{\partial F_{s',s}}{\partial s}=\frac{\partial G_{s',s}}{\partial s'}-\{F_{s',s},G_{s',s}\},
\end{equation}
where $\{-,-\}$ is the Poisson bracket. 
In this case, we can also apply \cref{thm:GKSmain} to the homogeneous Hamiltonian isotopy $\wh{\phi}$ and obtain a simple object $\tl{K}$ satisfying $\rMS(\tl{K})=\Lambda_{\wh{\phi}}$.
By using the map $q \colon (M \times \bR)^2 \times I^2 \to M^2 \times \bR_t \times I^2, (x_1,t_1,x_2,t_2,s',s) \mapsto (x_1,x_2,t_1-t_2,s',s)$, we also obtain an equivalence similar to \eqref{eq:equivalence}. 
Hence, we can define $K \in \SD(\bfk_{M^2 \times \bR \times I^2})$ by the condition $\rMS(K)=q_d q_\pi^{-1} (\Lambda_{\wh{\phi}})$ and $\cK \coloneqq P_l(K) \in \cD(M^2 \times I^2)$, which we call the \emph{sheaf quantization} of $(\phi_{s',s})_{(s',s)\in I^2}$.

\section{Strong stability result}

Our first result is the following.

\begin{theorem}\label{theorem:dist_equal}
    Let $H \colon T^*M \times I \to \bR$ be a timewise compactly supported  function. 
    Then one has $d_{\mathrm{isom}}(\cK^0_1,\cK^H_1)=d_{\mathrm{w\text{-}isom}}(\cK^0_1,\cK^H_1)$.
\end{theorem}

\begin{corollary}\label{corollary:stability}
    In the situation of \cref{theorem:dist_equal}, one has $d_{\mathrm{isom}}(\cK^0_1,\cK^H_1) \le \|H\|_{\mathrm{osc}}$.
\end{corollary}

\begin{proof}[Proof of \cref{theorem:dist_equal}]
    We only need to show that $d_{\mathrm{isom}}(\cK^0_1,\cK^H_1) \le d_{\mathrm{w\text{-}isom}}(\cK^0_1,\cK^H_1)$. 
    
    Let $\Tor$ be the full triangulated subcategory of $\cD(M^2)$ consisting of torsion objects $\{F \mid d_{\cD(M^2)}(F, 0)<\infty\}$. 
    Then, the Hom set of the localized category $\cD(M^2)/\Tor$ is computed as 
    \begin{equation}
        \Hom_{\cD(M^2)/\Tor}(F,G) \simeq \varinjlim_{c\to \infty}\Hom_{\cD(M^2)}(F,T_cG). 
    \end{equation}
    
    For the objects $\cK^0_1, \cK^H_1$ and $d \in \bR$, we have 
    \begin{equation}
        \Hom_{\cD(M^2)}(\cK^H_1,T_d \cK^H_1) \simeq 
        \Hom_{\cD(M^2)}(\cK^0_1,T_d \cK^0_1) \simeq 
        \begin{cases}
            \bfk & (d \ge 0) \\
            0 & (d <0).
        \end{cases}
    \end{equation}
    Hence, $\Hom_{\cD(M^2)/\Tor}(\cK^H_1,\cK^H_1) \simeq \bfk$ and the canonical morphism
    \begin{equation}
        \Hom_{\cD(M^2)}(\cK^H_1,T_d \cK^H_1) \to \Hom_{\cD(M^2)/\Tor}(\cK^H_1,\cK^H_1), \alpha' \mapsto \overline{\alpha'}
    \end{equation}
    is injective for $d \ge 0$. 
    
    Let $c>d_{\mathrm{w\text{-}isom}}(\cK^0_1,\cK^H_1)$.
    Then, by definition, there exist $a,b \in \bR_{\ge 0}$ with $c=a+b$ and morphisms $\alpha, \delta \colon \cK^0_1 \to T_a \cK^H_1$ and $\beta, \gamma \colon \cK^H_1 \to T_b \cK^0_1$ such that
    \begin{enumerate}
    	\renewcommand{\labelenumi}{$\mathrm{(\arabic{enumi})}$}
    	\item $\cK^0_1 \xrightarrow{\alpha} {T_a} \cK^H_1 \xrightarrow{{T_a} \beta} {T_{a+b}} \cK^0_1$ is equal to $\tau_{0,a+b}(\cK^0_1) \colon \cK^0_1 \to {T_{a+b}} \cK^0_1$,
    	\item $\cK^H_1 \xrightarrow{\gamma} {T_b} \cK^0_1 \xrightarrow{{T_b} \delta} {T_{a+b}} \cK^H_1$ is equal to $\tau_{0,a+b}(\cK^H_1) \colon \cK^H_1 \to {T_{a+b}} \cK^H_1$, and
    	\item $\tau_{a,2a}(\cK^H_1) \circ \alpha=\tau_{a,2a}(\cK^H_1) \circ \delta$ and $\tau_{b,2b}(\cK^0_1) \circ \beta=\tau_{b,2b}(\cK^0_1) \circ \gamma$.
    \end{enumerate}
    By the condition~(3), we have 
    \begin{align}
        \overline{\alpha}=\overline{\delta} & \in \Hom_{\cD(M^2)/\Tor}(\cK^0_1,\cK^H_1) \\
        \overline{\beta}=\overline{\gamma} & \in \Hom_{\cD(M^2)/\Tor}(\cK^H_1,\cK^0_1).
    \end{align}
    Hence, by the condition~(2), we get
    \begin{equation}
        1 = \overline{\delta} \circ \overline{\gamma} = \overline{\alpha} \circ \overline{\beta} \in \Hom_{\cD(M^2)/\Tor}(\cK^H_1,\cK^H_1).
    \end{equation}
    Since $T_b \alpha \circ \beta \in \Hom_{\cD(M^2)}(\cK^H_1,T_{a+b} \cK^H_1)$ is sent to $\overline{\alpha} \circ \overline{\beta}$ and $\tau_{0,a+b}(\cK^H_1)$ is sent to $1$, by the injectivity we obtain $T_b \alpha \circ \beta=\tau_{0,a+b}(\cK^H_1)$. 
    Thus, combining this with the condition~(1), we find that the pair $(\alpha,\beta)$ gives  $(a,b)$-isomorphisms for the pair $(\cK^0_1,\cK^H_1)$, which completes the proof.
\end{proof}

\begin{remark}
    By a similar argument, we can also improve \cite[Prop.~5.1]{AI22completeness}. 
    That is, we can replace the distance $d_{\mathrm{w\text{-}isom}}$ in the inequality
    \begin{equation}
        d_{\mathrm{w\text{-}isom}}(K^0_1,K^H_1) \le 2 \int_0^1 \|H_s\|_\infty ds \le 2\|H\|_{\mathrm{osc}}
    \end{equation}
    with a simpler and larger distance $d_{\mathrm{isom}}$, where $d_{\mathrm{w\text{-}isom}}$ (resp.\ $d_{\mathrm{isom}}$) is the distance on $\SD(\bfk_{M^2 \times \bR_t})$ defined by the weakly $(a,b)$-isomorphic (resp.\ $(a,b)$-isomorphic) relations for thickening kernel $\frakK_c = \bfk_{\Delta_M \times [-c,c]}$ (\cite{petit2020thickening})\footnote{Here, we regard $\frakK$ as a thickening kernel with respect to the operation $\bullet$ (see \cref{section:support_condition}).}.
\end{remark}

\section{Stability with support conditions}\label{section:support_condition}

For a closed subset $A$ of $T^*M$, we define a full subcategory $\cD_A(M)$ of $\cD(M)$ by 
\begin{equation}
    \cD_A(M) \coloneqq \{ F \in \cD(M) \mid \MS(F) \cap \{ \tau >0\} \subset \rho_t^{-1}(A) \},
\end{equation}
where $\rho_t \colon T^*M \times T^*_{\tau >0}\bR_t \to T^*M, (x,t;\xi,\tau) \mapsto (x;\xi/\tau)$.

For $K \in \SD(\bfk_{X \times M \times \bR})$ and $F \in \SD(\bfk_{M \times \bR_t})$, we set 
\begin{equation}
    K \bullet F \coloneqq \rR m_!(\tilde{q}_{1}^{-1}K \otimes \tilde{q}_{2}^{-1}F),
\end{equation}
where 
\begin{align}
    & \tilde{q}_{1} \colon X \times M \times \bR^2 \to X \times M \times \bR, 
    (x_1,x_2,t_1,t_2) \mapsto (x_1,x_2,t_1), \\
    & \tilde{q}_{2} \colon X \times M \times \bR^2 \to M \times \bR, 
    (x_1,x_2,t_1,t_2) \mapsto (x_2,t_2), \\
    & m \colon X \times M \times \bR^2 \to X \times \bR, 
    (x_1,x_2,t_1,t_2) \mapsto (x_1,t_1+t_2).
\end{align}
In what follows, we consider the following particular case.
Let $\cK^H \in \cD(M^2 \times I)$ be the sheaf quantization associated with a timewise compactly supported function $H \colon T^*M \times I \to \bR$ and $F \in \cD_A(M)$ with $A$ being a closed subset of $T^*M$.
Then we get $\cK^H \bullet F \in \cD(M \times I)$ and find that 
\begin{equation}
    \cK^H_s \bullet F \simeq (\cK^H \bullet F)|_{M \times \{s\} \times \bR_t} \in \cD_{\phi^H_s(A)}(M) \quad \text{for any $s \in I$}.
\end{equation}
We shall estimate the distance between $F \in \cD_A(M)$ and $\cK^H_1 \bullet F \in \cD_{\phi^H_1(A)}(M)$ up to translation.
See also \cref{remark:weak} for a more straightforward but weaker case. 

\begin{theorem}\label{theorem:metric_support}
    Let $A$ be a non-empty closed subset of $T^*M$ and $F \in \cD_A(M)$.
    Moreover, let $H \colon T^*M \times I \to \bR$ be a timewise compactly supported function. 
    Then for a continuous function $f\colon I\to \bR$, one has
	\begin{equation}\label{eq:bound}
	\begin{aligned}
	    & d_{\mathrm{isom}}(F,T_{-c} \cK^H_1 \bullet F) \\
		\le {} & \int_0^1 \lb \max \left\{ \max_{p \in \phi^H_s(A)}H_s(p), f(s) \right\} -\min \left\{ \min_{p \in \phi^H_s(A)}H_s(p), f(s)\right\} \rb ds
	\end{aligned}	
	\end{equation}
	where $c=\int_0^1 f(s) ds$. 
\end{theorem}

\begin{remark}\label{remark:metric_support}
    If we take $f\equiv 0$, we obtain 
    \begin{equation}
    \begin{aligned}
        & d_{\mathrm{isom}}(F, \cK^H_1 \bullet F) \\
		\le {} & \int_0^1 \lb \max \left\{ \max_{p \in \phi^H_s(A)}H_s(p), 0 \right\} -\min \left\{ \min_{p \in \phi^H_s(A)}H_s(p), 0\right\} \rb ds. 
    \end{aligned}
	\end{equation}

    Let $c \in \bR$ be a real number satisfying 
    \begin{equation}
        \int_0^1 \min_{p \in \phi^H_s(A)}H_s(p) ds\le c \le \int_0^1 \max_{p \in \phi^H_s(A)}H_s(p) ds. 
    \end{equation}
    Then we can take $f$ such that $c=\int_0^1 f(s) ds$ and 
    \begin{equation}
        \min_{p \in \phi^H_s(A)}H_s(p) \le f(s)\le \max_{p \in \phi^H_s(A)}H_s(p)
    \end{equation}
    for any $s \in I$. 
    Hence, by \cref{theorem:metric_support}, we get
	\begin{equation}
		d_{\mathrm{isom}}(F, T_{-c} \cK^H_1 \bullet F)
		\le \int_0^1 \lb \max_{p \in \phi^H_s(A)}H_s(p) - \min_{p \in \phi^H_s(A)}H_s(p) \rb ds.
	\end{equation}
\end{remark}

For simplicity, we introduce a symbol for the right-hand side of \eqref{eq:bound}. 
For a function $H \colon T^*M \times I \to \bR$, a function $f \colon I \to \bR$, and a non-empty closed subset $A$ of $T^*M$, we set 
\begin{equation}
    B(H,f,A) 
    \coloneqq 
    \int_0^1 \lb \max \left\{ \max_{p \in \phi^H_s(A)}H_s(p), f(s) \right\} -\min \left\{ \min_{p \in \phi^H_s(A)}H_s(p), f(s) \right\} \rb ds. 
\end{equation}

\begin{proof}[Proof of \cref{theorem:metric_support}]
	Let $\varepsilon >0$.
	We can take a smooth family $(\rho_{a,b}\colon \bR\to \bR)_{a,b}$ of smooth functions parametrized by $a,b\in \bR$ with $a\le b$ such that 
	\begin{enumerate}
	\renewcommand{\labelenumi}{$\mathrm{(\arabic{enumi})}$}
	    \item $\rho_{a,b}(y)=y$ on a neighborhood of $[a,b]$,
	    \item $a-\varepsilon \le \inf_y \rho_{a,b}(y) < \sup_y \rho_{a,b}(y) \le b+\varepsilon$. 
	\end{enumerate}
    Recall that $I$ denotes an open interval containing the closed interval $[0,1]$.
	We take smooth functions $M, m \colon I\to \bR$ satisfying 
	\begin{equation}
	    \max_{p \in \phi^H_s(A)}H_s(p) +\frac{\varepsilon}{2} \le M(s) \le \max_{p \in \phi^H_s(A)}H_s(p)  +\varepsilon
	\end{equation}
    and
    \begin{equation}
        \min_{p \in \phi^H_s(A)}H_s(p) -\varepsilon \le m(s) \le \min_{p \in \phi^H_s(A)}H_s(p)-\frac{\varepsilon}{2}.
    \end{equation}
	Fix $R>0$ sufficiently large so that $R> \max_{p,s}H_s(p) -\min_{p,s}H_s(p) +2\varepsilon$.
	Define $a(s', s)\coloneqq m(s)-Rs', b(s',s)\coloneqq M(s)+Rs'$ for $(s',s)\in I^2$. 
	We may assume that $I\subset (-\frac{\varepsilon}{2R},+\infty)$ by taking $I$ smaller if necessary, and hence that 
    \begin{equation}
        a(s', s) \le \min_{p \in \phi^H_s(A)}H_s(p) \le \max_{p \in \phi^H_s(A)}H_s(p) \le b(s',s)
    \end{equation}
    for all $(s',s)\in I^2$. 
    Take a smooth function $\tilde{f}\colon I\to \bR$ such that $\|\tilde{f}-f\|_{C^0} \le \varepsilon$. 
    By shrinking $I$, we may assume that $\bigcup_{s} \supp(H_s)$ is relatively compact.
    Then we can also take a compactly supported smooth cut-off function $\chi \colon T^*M\to [0,1]$ such that $\chi\equiv 1$ on a neighborhood of $\bigcup_{s} \supp(H_s)$.
    Using these functions, we define a function $G=(G_{s',s})_{s',s \in I^2} \colon T^*M \times I^2 \to \bR$ by 
    \begin{equation}
        G_{s',s}\coloneqq \left(\rho_{a(s',s),b(s',s)}\circ H_s - (1-s') \tilde{f}(s)\right) \chi.
    \end{equation}
	A 2-parameter family $(\phi_{s',s})_{(s',s)\in I^2}$ of Hamiltonian diffeomorphisms is determined by $\phi_{s',0}=\id_{T^*M} $ and $\frac{\partial \phi_{s',s} }{\partial s}\circ \phi_{s',s}^{-1}=X_{G_{s',s}}$, where $X_{G_{s',s}}$ is the Hamiltonian vector field corresponding to the function $G_{s',s}$. 
	Note that $G_{1,s}=H_s$ and $\phi_{s',s}$ is independent of $s'$ on a neighborhood $U$ of $A$. 
	Moreover, we have 
	\begin{equation}
	\begin{aligned}
        & \int_0^1 \lb \max_{p \in T^*M} G_{0,s}(p) - \min_{p \in T^*M} G_{0,s}(p) \rb ds \\
        \le {} & 
        \int_0^1 \lb \max_{p \in T^*M} \left( \rho_{m(s),M(s)} \circ H_s(p) - \tilde{f}(s) \right)\chi(p) - \min_{p \in T^*M} \left( \rho_{m(s),M(s)} \circ H_s(p) - \tilde{f}(s) \right) \chi(p) \rb 
        \\ 
        \le {} & 
        \int_0^1 \lb \max \left\{ \max_{p \in \phi^H_s(A)} \left( H_s(p) - \tilde{f}(s) \right), 0 \right\} -\min \left\{ \min_{p \in \phi^H_s(A)} \left( H_s(p) - \tilde{f}(s) \right), 0\right\} \rb ds +2\varepsilon \\
        = {} & 
        \int_0^1 \lb \max \left\{ \max_{p \in \phi^H_s(A)} H_s(p), \tilde{f}(s) \right\} -\min \left\{ \min_{p \in \phi^H_s(A)} H_s(p), \tilde{f}(s) \right\} \rb ds +2\varepsilon \\
        \le {} &
        B(H,f,A) +4\varepsilon.
	\end{aligned}
	\end{equation}
    For $s' \in I$, we set $G_{s'} \coloneqq G_{s',\bullet} \colon T^*M \times I \to \bR$.
    Then, by \cref{corollary:stability} and the natural inequality for the distance with respect to functorial operations, we obtain 
    \begin{equation}
        d_{\mathrm{isom}}(F, \cK^{G_0}_1 \bullet F) = d_{\cD(M)}(\cK^0_1 \bullet F, \cK^{G_0}_1 \bullet F) 
        \le 
        B(H,f,A) +4\varepsilon.
    \end{equation}

    We set $\tilde{c}(s)\coloneqq  \int_0^s \tilde{f}(t) dt$ and claim that $\cK^{G_0}_1 \bullet F \simeq T_{-\tilde{c}(1)}\cK^{G_1}_1 \bullet F$. 
    By the result recalled in \cref{subsec:SQ_ham}, we can construct the sheaf quantization $\cK \in \cD(M^2 \times I^2)$ of the 2-parameter family of diffeomorphisms $(\phi_{s',s})_{s',s}$. 
    We shall use the same notation as in \cref{subsec:SQ_ham}.
	Then, $F_{s',0}=0$ and $F_{\bullet,s}|_{\phi_{1,s}(U)\times I}\colon \phi_{1,s}(U)\times I\to \bR, (p,s')\mapsto F_{s',s}(p)$ is locally constant for each $s$. 
    By \eqref{eq:another_function}, we find that $\frac{\partial F_{s',s}}{\partial s}=\frac{\partial G_{s',s}}{\partial s'}=\tilde{f}(s)$ on $\bigcup_s \phi_{1,s}(U)\times I\times \{s\}$ and that $F_{s',s} = \int_0^s \tilde{f}(t) dt=\tilde{c}(s)$ there. 
    We define $\cH \coloneqq \cK \bullet F \in \cD(M \times I^2)$. 
    Then, by the microsupport estimate, we have 
    \begin{equation}
    \begin{aligned}
        \MS(\cH) 
        \subset 
        \left\{ \lb \widehat{\phi}_{s',s}(x,t;\xi,\tau), (s';-\tau \tilde{c}(s)), (s;-\tau G_{s',s}(\phi_{s',s}(x;\xi/\tau))) \rb \; \middle| \right. \\ 
        \left. (x,t;\xi,\tau) \in \rMS(F), s',s \in I \right\} 
        \cup 0_{M \times \bR \times I^2}.
    \end{aligned}  
    \end{equation}
    Hence, $M \times \bR \times I \times \{1\}$ is non-characteristic for $\cH$ and we get 
    \begin{equation}
    \begin{aligned}
        \MS(\cH|_{M \times \bR \times I \times \{1\}})
        \subset & 
        \left\{ \lb \widehat{\phi}_{s',1}(x,t;\xi,\tau), (s';- \tilde{c}(1) \tau) \rb \; \middle| \right. \\
        & \quad \left. (x,t;\xi,\tau) \in \rMS(F), s' \in I \right\} 
        \cup 0_{M \times \bR \times I}.
    \end{aligned}       
    \end{equation}
    Define a diffeomorphism $\varphi \colon M \times \bR \times I \simto M \times \bR \times I, (x,t,s') \mapsto (x,t-\tilde{c}(1)s', s')$. 
    Then we have $\MS(\varphi_* \cH|_{M \times \bR \times I \times \{1\}}) \subset T^*(M \times \bR) \times 0_{I}$, which shows $\varphi_* \cH|_{M \times \bR \times I \times \{1\}}$ is the pull-back of a sheaf on $M \times \bR$ by \cite[Prop.~5.4.5]{KS90}.
    In particular, 
    \begin{equation}
    \begin{aligned}
        \cK^{G_0}_1 \bullet F 
        & = 
        \cH|_{M \times \bR \times \{0\} \times \{1\}} \\
        & \simeq 
        (\varphi_* \cH|_{M \times \bR \times I \times \{1\}})|_{\{ s'=0\}} \\
        & \simeq 
        (\varphi_* \cH|_{M \times \bR \times I \times \{1\}})|_{\{ s'=1\}} \\
        & \simeq 
        T_{-\tilde{c}(1)}\cH|_{M \times \bR \times \{1\} \times \{1\}} 
        =
        T_{-\tilde{c}(1)} \cK^{G_1}_1 \bullet F. 
    \end{aligned}
    \end{equation}

    Since $|c-\tilde{c}(1)| \le \varepsilon$, we have $d_{\mathrm{isom}}(T_{-c}\cK^H_1 \bullet F, T_{-\tilde{c}(1)}\cK^H_1 \bullet F) \le \varepsilon$.
    Combining the result above and noticing $G_1=H$, we obtain 
    \begin{equation}
    \begin{aligned}
        d_{\mathrm{isom}}(F, T_{-c} \cK^H_1 \bullet F)
        & \le 
        d_{\mathrm{isom}}(F, T_{-\tilde{c}(1)} \cK^H_1 \bullet F) + \varepsilon \\
        & = 
        d_{\mathrm{isom}}(F, \cK^{G_0}_1 \bullet F) + \varepsilon \\
        & \le 
        B(H,f,A) +5\varepsilon.
    \end{aligned}
    \end{equation}
    Since $\varepsilon>0$ is arbitrary, this completes the proof.
\end{proof}

\begin{remark}\label{remark:weak}
    Under the same assumption as in \cref{theorem:metric_support}, we can prove the weaker result 
    \begin{equation}
	   d_{\mathrm{w\text{-}isom}}(F,T_{-c} \cK^H_1 \bullet F) \le B(H,f,A) 
	\end{equation}
    more straightforwardly, without the 2-parameter family, as follows. 

    We set $\cH \coloneqq \cK^H \bullet F \in \cD(M \times I)$. 
    Then we have $\cH|_{M \times \bR_t \times \{0\}} \simeq F$ 
	and $\cH|_{M \times \bR_t \times \{1\}} \simeq \cK^H_1 \bullet F$.
    Moreover, by the microsupport estimate, we find that 
	\begin{equation}
	    \MS(\cH) 
	    \subset T^*M \times \left\{ (t,s;\tau,\sigma) \;\middle|\;  -\max_{p \in \phi^H_s(A)} H_s(p) \cdot \tau \le \sigma \le -\min_{p \in \phi^H_s(A)} H_s(p) \cdot \tau \right\}.
	\end{equation}
    Let $\varepsilon>0$ and take a smooth function $\tilde{f}\colon I\to \bR$ such that $\|\tilde{f}-f\|_{C^0} \le \varepsilon$. 
    We define a function $\tilde{c} \colon I \to \bR$ by $\tilde{c}(s) \coloneqq \int_0^s \tilde{f}(s')ds'$ and    
    a function $\varphi \colon M \times \bR_t \times I  \to M \times \bR_t \times I$ by $\varphi(x,t,s) \coloneqq (x,t-\tilde{c}(s),s)$.
    Then we have $\varphi_* \cH|_{M \times \bR_t \times \{0\}} \simeq F$, $\varphi_* \cH|_{M \times \bR_t \times \{1\}} \simeq T_{-\tilde{c}(1)}\cK^H_1 \bullet F$, and     
    \begin{equation}
    \begin{aligned}
        & \MS(\varphi_* \cH) \\
	    \subset {} &
        T^*M \times \left\{ (t,s;\tau,\sigma) \;\middle|\;  - \lb \max_{p \in \phi^H_s(A)} H_s(p) -\tilde{f}(s) \rb \cdot \tau \le \sigma \le - \lb \min_{p \in \phi^H_s(A)} H_s(p) - \tilde{f}(s) \rb \cdot \tau \right\}.
    \end{aligned} 
    \end{equation}
    Note that we may have $\max_{p \in \phi^H_s(A)} H_s(p) -\tilde{f}(s) < 0$ and $\min_{p \in \phi^H_s(A)} H_s(p) - \tilde{f}(s) >0$ in general.
    By applying \cite[Prop.~3.9]{AI22completeness}, we obtain 
    \begin{equation}
    \begin{aligned}
        & d_{\mathrm{w\text{-}isom}}(F,T_{-\tilde{c}(1)}\cK^H_1 \bullet F) \\
        = {} &  
        d_{\mathrm{w\text{-}isom}}(\varphi_* \cH|_{M \times \bR_t \times \{0\}}, \varphi_* \cH|_{M \times \bR_t \times \{1\}}) \\
        \le {} & 
        \int_0^1 \lb \max \left\{ \max_{p \in \phi^H_s(A)} \left( H_s(p) - \tilde{f}(s) \right), 0 \right\} -\min \left\{ \min_{p \in \phi^H_s(A)} \left( H_s(p) - \tilde{f}(s) \right), 0\right\} \rb ds \\
        = {} &
        \int_0^1 \lb \max \left\{ \max_{p \in \phi^H_s(A)} H_s(p), \tilde{f}(s) \right\} -\min \left\{ \min_{p \in \phi^H_s(A)} H_s(p), \tilde{f}(s) \right\} \rb ds \\
        \le {} & 
        \int_0^1 \lb \max \left\{ \max_{p \in \phi^H_s(A)}H_s(p), f(s) \right\} -\min \left\{ \min_{p \in \phi^H_s(A)}H_s(p), f(s)\right\} \rb ds + 2\varepsilon.
    \end{aligned}
    \end{equation}
    Hence, we have 
    \begin{equation}
    \begin{aligned}
        & d_{\mathrm{w\text{-}isom}}(F, T_{-c}\cK^H_1 \bullet F) 
        \le d_{\mathrm{w\text{-}isom}}(F, T_{-\tilde{c}(1)} \cK^H_1 \bullet F) + \varepsilon \\
        \le {} &
        \int_0^1 \lb \max \left\{ \max_{p \in \phi^H_s(A)}H_s(p), f(s) \right\} -\min \left\{ \min_{p \in \phi^H_s(A)}H_s(p), f(s)\right\} \rb ds + 3\varepsilon, 
    \end{aligned}
    \end{equation}
    which completes the proof.    
\end{remark}

For a timewise compactly supported function $H \colon T^*M \times I \to \bR$ and a non-empty closed subset $A$ of $T^*M$, we set 
\begin{equation}
    \|H\|_{\mathrm{osc},A} \coloneqq \int_0^1 \left(\max_{p \in A} H_s(p) - \min_{p \in A} H_s(p) \right) ds.
\end{equation}

\begin{proposition}
    Let $A$ be a non-empty closed subset of $T^*M$ and $F \in \cD_A(M)$.
    Moreover, let $H \colon T^*M \times I \to \bR$ be a timewise compactly supported function. 
    Then there exists $c \in \bR$ such that 
	\begin{equation}
		d_{\mathrm{isom}}(F,T_{-c}\cK^H_1 \bullet F)
		\le \|H\|_{\mathrm{osc},A}.
	\end{equation}
\end{proposition}

\begin{proof}
    Using the technique in the proof of \cite[Theorem~1.3]{usher2015observations}, one can construct a function $H'$ such that $\phi^H_1=\phi^{H'}_1$ and 
	\begin{equation*}
	\begin{aligned}
	    & \int_0^1 \lb \max_{p \in A}H_s(p) - \min_{p \in A}H_s(p) \rb ds \\
	    ={} & \int_0^1 \lb \max_{p \in \phi^{H'}_s(A)}H'_s(p) - \min_{p \in \phi^{H'}_s(A)}H'_s(p) \rb ds.
	\end{aligned}
	\end{equation*}
	By \cite[Proposition~5.7]{AI22completeness}, $\phi^H_1=\phi^{H'}_1$ implies $\cK^H_1 \simeq \cK^{H'}_1$. 
	Hence, the result follows from \cref{theorem:metric_support} (see also \cref{remark:metric_support}).
\end{proof}

\printbibliography

\noindent Tomohiro Asano: 
Research Institute for Mathematical Sciences, Kyoto University, Kitashirakawa-Oiwake-Cho, Sakyo-ku, 606-8502, Kyoto, Japan.

\noindent \textit{E-mail address}: \texttt{tasano@kurims.kyoto-u.ac.jp}, \texttt{tomoh.asano@gmail.com}

\medskip

\noindent Yuichi Ike:
Institute of Mathematics for Industry, Kyushu University, 744 Motooka Nishi-ku, 819-0395, Fukuoka, Japan.

\noindent
\textit{E-mail address}: \texttt{ike@imi.kyushu-u.ac.jp}, \texttt{yuichi.ike.1990@gmail.com}

\end{document}